\documentclass{article}
\usepackage[utf8]{inputenc}
\usepackage{amsfonts}
\usepackage{amsmath}
\usepackage[margin=1in]{geometry}
\usepackage{amsthm}
\usepackage{comment}
\usepackage{xcolor}
\usepackage{color}
\usepackage{amssymb}
\newtheorem{theorem}{Theorem}[section]
\newtheorem{remark}[theorem]{Remark}

\newtheorem{proposition}[theorem]{Proposition}
\newtheorem{corollary}[theorem]{Corollary}
\newtheorem{lemma}[theorem]{Lemma}
\newtheorem{assumption}[theorem]{Assumption}
\usepackage[normalem]{ulem}

\newcommand{\sy}[1]{{\color{blue} #1}}

\title{Robust Decentralized Control of Coupled Systems via Risk Sensitive Control of Decoupled or Simple Models with Measure Change
}
\author{Zachary Selk\footnote{Department of Mathematics and Statistics, Queen's University. zachary.selk@queensu.ca}\and Serdar Y\"uksel\footnote{Department of Mathematics and Statistics, Queen's University. yuksel@queensu.ca} }
\date{\today}

\begin{document}

\maketitle
\begin{abstract}
 Decentralized stochastic control problems with local information involve problems where multiple agents and subsystems which are coupled via dynamics and/or cost are present. Typically, however, the dynamics of such couplings is complex and difficult to precisely model, leading to questions on robustness in control design. Additionally, when such a coupling can be modeled, the problem arrived at is typically challenging and non-convex, due to decentralization of information. In this paper, we develop a robustness framework for optimal decentralized control of interacting agents, where we show that a decentralized control problem with interacting agents can be robustly designed by considering a risk-sensitive version of non-interacting agents/particles. This leads to a tractable robust formulation where we give a bound on the value of the cost function in terms of the risk-sensitive cost function for the non-interacting case plus a term involving the ``strength" of the interaction as measured by relative entropy. We will build on Gaussian measure theory and an associated variational equality. A particular application includes mean-field models consisting of (a generally large number of) interacting agents which are often hard to solve for the case with small or moderate numbers of agents, leading to an interest in effective approximations and robustness. By adapting a risk-sensitivity parameter, we also robustly control a non-symmetrically interacting  problem with mean-field cost by one which is symmetric with a risk-sensitive criterion, and in the limit of small interactions, show the stability of optimal solutions to perturbations. %The discrete-time setup is also studied where change of measure arguments have been adapted to this setup, building on Witsenhausen's prior work. Several examples are provided.

\end{abstract}
\section{Introduction}

 Decentralized stochastic control involves problems where multiple agents and subsystems are coupled via dynamics and/or cost and agents have local information. Typically,  the dynamics of such couplings is difficult to precisely model, leading to questions on robustness. Additionally, when such a coupling can be modelled, the problem arrived at is typically non-convex and difficult to solve due to decentralization in information for which classical tools in stochastic control, such as dynamic programming are no longer directly applicable (see e.g. \cite{wit68,YukselWitsenStandardArXiv}). 

  In this paper, we develop a framework for risk-sensitive control for the study of decentralized stochastic control problems with uncertain interactions among a finite or infinite collection of interacting particles. 
  
  Risk-sensitive stochastic control has been shown to be an effective paradigm for robust stochastic control, with intrinsic connections between large deviations and relative entropy; seminal studies in the stochastic domain include \cite{dupuis2000robust,dupuis2000kernel,anantharam2017variational}. Rather than minimizing an average ``risk neutral" cost $N(\gamma):=E[C(\gamma)]$, one considers the ``risk sensitive" cost $S(\gamma):=\log E[e^{C(\gamma)}]$. By punishing deviations more harshly through exponentiation, minimizing $S$ leads to policies that are ``robust" to perturbations of the model. More precisely, for each control $\gamma$, and under suitable conditions on $C$, we have the variational equality
\begin{align}\label{DVVarEq}
\log E_{\mu_0}[e^{C(\gamma)}]=\sup_{\mu \sim \mu_0}\{E_\mu[C(\gamma)]-D_{KL}(\mu||\mu_0)\},
\end{align}
where $D_{KL}$ is the KL-divergence or relative entropy and $\sim$ denotes equivalence (mutual absolute continuity) of measures. In \cite{dupuis2000robust}, the authors related risk neutral control of a system driven by ``true noise" with a risk sensitive control of a system driven by ``nominal noise". That is, if $\mu_0$ is the law of some ``nominal" noise driving the system and $\mu$ is the law of the actual noise of the system, then we have the bound
\[E_\mu[C(\gamma)]\leq \log E_{\mu_0}[e^{C(\gamma)}]+D_{KL}(\mu||\mu_0).\]

We note that for risk sensitive control problems, there is also a direct relationship with a zero-sum formulation between the controller and disturbance \cite{jacobson1973optimal,basbern,bacsar2021robust}. These approaches (to robustness) are complementary; the former uses a variational approach and the latter typically builds on the direct equivalence of the associated Hamilton-Jacobi-Bellman equations.
 
 We study decentralized stochastic control problems, where the (sometimes poorly modeled) interaction among agents leads to two types of challenges: (i) When the mathematical model is known precisely, lack of convexity arising from decentralized information structure makes the analysis challenging \cite{charalambous2016decentralized,huang2022general,jackson2023approximately,saldiyukselGeoInfoStructure,wit68,YukselSaldiSICON17}), and (ii) when the mathematical model is not known precisely, a natural robustness question arises.

In the context of our paper, a particular setup involves mean-field theory. Mean field theory \cite{carmona2016mean,CainesMeanField3,CainesMeanField2,lacker2015mean,LyonsMeanField} models control problems when there are (a generally large number of) interacting agents. One key observation in mean field theory is that, under some conditions, for large enough systems of agents one can approximate all interactions through a mean field coupling term - that is, each agent interacts with the average behavior of all agents. However, due to this coupling, mean field control problems are often very hard to solve when the number of agents is not large or when the agents apply non-identical policies which prohibits concentration results for the asymptotically large number of agents regime (see \cite{sanjari2020optimality,sanjarisaldiy2024optimality,sanjarisaldiyuksel2023CTMF} on possible non-optimality of symmetric policies for finitely many agents but their optimality in the infinite limit). This leads to an interest in effective approximations and robust solutions.

Our approach builds on defining equivalent information structures via change of measure arguments, but unlike those in the literature starting with \cite{benevs1971existence}, we approach the equivalent information structures with the aim of attaining robustness to solutions of complex decentralized stochastic control problems. The common thread for change of measure arguments (see \cite{sanjari2021optimality} for a detailed review) is an absolute continuity condition involving a family of conditional probabilities on the measurement variables (given the past history or state realizations) with respect to a common reference measure. By a different approach and motivation, in this paper, we develop a framework for risk-sensitive control for the study of decentralized stochastic control problems with uncertain interactions among a finite or infinite collection of interacting particles. We will utilize tools from Gaussian measure theory and an associated variational equality.

%We will present a particular application to mean-field models, where we will establish stronger results using symmetry: Mean-field games (see e.g., \cite{CainesMeanField2,CainesMeanField3,LyonsMeanField,armona2016mean, lacker2015mean}) can be viewed as limit models of symmetric non-zero-sum non-cooperative finite player games with a mean-field interaction. Mean-field teams study the case with agents having identical cost criteria, but different local information. In particular, under decentralized information structures, mean-field problems generally correspond to dynamic team problems with non-classical information structures. The existence of Nash equilibria (and hence, person-by-person-optimal solutions for team problems) has been established for mean-field games in \cite{LyonsMeanField, carmona2016mean, lacker2015mean}, among others.

In the decentralized information setup, we note that the robustness relationship presented in our paper cannot be obtained directly through an equivalence associated with a Hamilton Jacobi Bellman equation arrived at via a game theoretic formulation, as such an optimality condition is currently not available due to the decentralized information structure and does not typically arise (though optimality conditions can be obtained for convex teams) in team theoretic problems under information constraints when there are finitely many agents.

Our paper also motivates the study of risk-sensitive decentralized stochastic control, which has been studied in the recent literature \cite{bensoussan2017risk,moon2016linear,saldi2020approximate,saldi2022partially,tembine2013risk}.
 
{\bf Main Results.} The variational equality (\ref{DVVarEq}) can also be applied where the nominal reference measure $\mu_0$ is the law of non-interacting agents and $\mu$ is the law of interacting agents. In this paper, we propose to use this variational equality to approximate a risk neutral interacting agent mean field control problem by a risk sensitive non-interacting control problem. Theorem \ref{theorem:main} presents a robustness result under a more relaxed energy condition but with a mean-field coupling, and Theorem \ref{theorem:main-Gaussian-measure} presents the robustness result for decentralized systems under arbitrary coupling but with a more stringent energy condition. 

%We do this by perturbing the measure on the agents (i.e., the strategic path measure), rather than only the measure on the noise. On the choice of reference measure, although the risk sensitive costs are equivalent, choosing the reference measure associated to the nominal model to be the law of the noise or the law of the solution provides slightly different behavior in the true model. By choosing the reference measure to be the law of the solution rather than the law of the noise, we allow for a more natural perturbation at the cost of a slightly worse robustness bound and less allowable perturbations. 

By adapting a risk-sensitivity parameter, we also (i) robustly control a non-symmetrically interacting  problem with mean-field cost by one which is symmetric with a risk-sensitive criterion, and (ii) in the limit of small interactions, show the stability of optimal solutions to perturbations. 
%The discrete-time setup is also studied where change of measure arguments have been adapted to this setup, building on Witsenhausen's prior work. Several examples are provided.

We also discuss optimal risk sensitivity. By introducing a parameter $\alpha\in \mathbb R$ which quantifies the ``amount" of risk aversion, we can get the bound
\begin{equation}\label{eq:main-result-alpha}
E_\mu[C(\gamma)]\leq \frac{1}{\alpha}\log E_{\mu_0}[e^{\alpha C(\gamma)}]+\frac{1}{\alpha}D_{KL}(\mu||\mu_0).\end{equation}
By optimizing over $\alpha$ we obtain Theorem \ref{thmoptimizeparam}. 
%\begin{theorem}
%    We have that
%    \begin{equation*}
%        E_\mu[C(\gamma)]\leq E_{\hat\mu}[C(\gamma)],
%    \end{equation*}
%    where $\frac{d\hat \mu}{d\mu_0}\propto e^{\alpha^\ast %C(\gamma)}$ and $\alpha^\ast$ is a minimizer of the right %hand side of equation \eqref{eq:main-result-alpha}.
%\end{theorem}

As a corollary, we show how we can recover the true model in the small interaction limit by selecting $\alpha$ appropriately. We additionally show how we can approximate a non-symmetric problem by a symmetric problem. See Section \ref{sec:optimal-risk-sensitivity} for details and additional implications, including on stability of risk-sensitive optimal solutions to small interactions.

\section{Problem Setup}
We consider two system models, one with a perturbation of the measure on the agents and the other with a perturbation on the noise (similar to \cite{dupuis2000robust}). \\

{\bf System Model A.} Consider a system of interacting controlled differential systems (particles) where each $X^m_t; t \in [0,T], m=1,\cdots, N$ is $\mathbb{R}^n$-valued
\begin{align}
    d\tilde X^1_t&=(b_1(\tilde X^1_t,U^1_t) +v_t^1) dt+\sigma_1(\tilde X^1_t) dB^1_t \nonumber \\
    &\vdots \nonumber \\
     d\tilde X^N_t&=(b_N(\tilde X_t^N,U^N_t) +v_t^N) dt+\sigma_N(\tilde X^N_t) dB^N_t, 
     \label{trueModel}
\end{align}
with $X^{[1,N]}_0=x^{[1,N]}_0$ and where the $v^i$ (which may be incorrectly modelled) depend on the entire ensemble $(\tilde X^1,...,\tilde X^N)$ so that $v^i_t$ is $\sigma(\{X^{[1,N]}_{[0,t]}\})$-measurable,
 and satisfies $E\left[\int_0^T v_t^2 dt\right]<\infty.$ With $\mathbb{U}$ a compact metric space, $U^m_t$ is $\mathbb{U}$ valued for each $m =1,\cdots, N$ and $t \in [0,T]$
 \\\\
 {\bf System Model B.} Consider a system of interacting controlled differential systems (particles)
\begin{align}
    d\tilde X^1_t&=(b_1(\tilde X^1_t,U^1_t) +\sigma_1(\tilde X^1_t)v_t^1) dt+\sigma_1(\tilde X^1_t) dB^1_t \nonumber \\
    &\vdots \nonumber \\
     d\tilde X^N_t&=(b_N(\tilde X_t^N,U^N_t) +\sigma_N(\tilde X^N_t)v_t^N) dt+\sigma_N(\tilde X^N_t) dB^N_t, \label{trueModel-B}
 \end{align}
with $X^{[1,N]}_0=x^{[1,N]}_0$ and where the $v^i$ depend on the entire noise $(B^1,...,B^N)$ so that $v^i_t$ is $\sigma(\{B^{[1,N]}_{[0,t]}\})$-measurable,
 and satisfies $E\left[\int_0^T v_t^2 dt\right]<\infty.$ 

%The models above will lead to slightly different perturbations of a nominal SDEs to be defined shortly. The first model has the more natural perturbation $v_t$ and the second one has the perhaps more unnatural $\sigma(X)v_t$. \\\\

Here $B^i$ are i.i.d. Brownian motions, $U^i_t$ are control action variable processes generated under control policies $\gamma^i \in \Gamma^i$, where $\Gamma^i$ is the set of admissible policies:

In the setup to be considered the admissible policies may be decentralized or centralized. In the decentralized setup, we view each $X^i$, $1 \leq i \leq N$, to be controlled locally (subject to the coupling), such that at each time $U^i_t$ is a possibly ${\cal P}(\mathbb{U})$-valued $\sigma(X^i_{[0,t]})$-measurable control. In the centralized setup, we assume that $U^i_t$ is a possibly ${\cal P}(\mathbb{U})$-valued $\sigma(X^{1,N}_{[0,t]})$-measurable control.  

Our focus will be on the decentralized setup, as the results are more consequential given the difficulties associated with optimal decentralized stochastic control with local information. 

As noted in \cite[Chapter 2]{arapostathis2012ergodic} (see also \cite[Section 4]{pradhanyuksel2023DTApprx}) under such (decentralized or centralized) feedback policies, the question on the existence of strong solutions leads to several subtleties. Accordingly, instead of feedback policies, one may also consider policies which are probability measure-valued and adapted (as measure-valued controls) to the local driving noise so that the policies are non-anticipative: thus, one requires that $\{(U^i_s,B^i_s), s \leq t\}$ be independent of $B^i_r - B^i_t, r > t$ for all $t \in [0,T]$, and also independent of $B^m_t, t \in [0,T]$ for $i \neq m, m=1,2,\cdots,N$. Accordingly, in the above, standard Lipschitz growth conditions \cite[Theorem 2.2.4]{arapostathis2012ergodic} (see also \cite[Section 4]{pradhanyuksel2023DTApprx}) are to be imposed on $b_i,\sigma_i$ to ensure the existence of strong solutions to the stochastic differential equations above under the policies considered. This is needed for change of measure arguments to be made in our analysis (so that policies can also be expressed as a function of the driving noise validating the change of measure equations to be studied). Accordingly, we impose the following assumptions on the drift $b_i$ and the diffusion matrices $\sigma_i$. 
\begin{itemize}
\item[(A1)]
\emph{Lipschitz continuity:\/}
The function
$\sigma_i: \mathbb{R}^{d}\to\mathbb{R}^{d\times d}$,
$b_i: \mathbb{R}^d \times\mathbb{U}^i \to\mathbb{R}^d$ are Lipschitz continuous in $x$ (uniformly with respect to the control action for $b$). In particular, for some constant $C>0$
depending on $R>0$, we have
\begin{equation*}
\| b_i(x,u) - b_i(y, u)\|^2 + \|\sigma_i(x) - \sigma_i(y)\|^2 \,\le\, C\,\|x-y\|^2
\end{equation*}
for all $x,y\in \mathbb{R}^n$ and $u \in\mathbb{U}^i$, where $\| \sigma_i\| := \sqrt{\mathrm{Trace}(\sigma_i\sigma_i^T)}$. 
We also assume that $b$ is jointly continuous in $(x,u)$.

\medskip
\item[(A2)]
\emph{Affine growth condition:}
$b_i$ and $\sigma_i$ satisfy a global growth condition of the form
\begin{equation*}
\sup_{u \in\mathbb{U}^i } |\langle b_i(x,u),x\rangle| +  \|\sigma_i(x)\|^{2} \leq  C_0 (1 + \|x\|^{2})  \qquad \forall x \in \mathbb{R}^{n},
\end{equation*}
for some constant $C_0>0$.
\end{itemize}

We consider the following cost criteria.\\  

{\bf General Cost with $N < \infty$.}

\begin{equation*}
    J_N(\underline \gamma_N)=E^{\underline \gamma_N}_{\mu_0^N} \left[c_N(X^{[1,N]},U^{[1,N]})\right],
\end{equation*}

with
\begin{equation}\label{generalCostFnD}
c_N(X^{[1,N]},U^{[1,N]}) = \int_0^{T} g(X_s^{[1,N]},U_s^{[1,N]}) ds
\end{equation}
where $\underline \gamma_N=(\gamma^1,...,\gamma^N)$ with $U^i_t=\gamma^i(X^i_{[0,t]})$. \\\\

We will also consider a more specialized mean field symmetric cost. This will allow for symmetric interactions in our system and more relaxed conditions.\\

{\bf Mean-Field Symmetric Cost with $N < \infty$ or $N=\infty$.}

Consider
\begin{equation*}
    J_N(\underline \gamma_N)=E^{\underline \gamma_N}_{\mu_0^N} \left[c_N(X^{[1,N]},U^{[1,N]})\right],
\end{equation*}

with
\begin{equation}\label{generalCostFnMF}
c_N(X^{[1,N]},U^{[1,N]}) = \frac{1}{N} \sum_{i=1}^N c\left(X^i, U^i, \frac{1}{N}\sum_{m=1}^N X^m,\frac{1}{N}\sum_{m=1}^N U^m\right)
\end{equation}

where
\[c\left(X^i, U^i, \frac{1}{N}\sum_{m=1}^N X^m,\frac{1}{N}\sum_{m=1}^N U^m\right) = \int_0^{T} g(X^i_s, U^i_s, \frac{1}{N}\sum_{m=1}^N X^m_s, \frac{1}{N}\sum_{m=1}^N U^m_s) ds \]

where $\underline \gamma_N=(\gamma^1,...,\gamma^N)$ with $U^i_t=\gamma^i(X^i_{[0,t]})$. We are also interested in the limit
\begin{equation}\label{eq:J-def}
    J(\underline \gamma)=\limsup_{N\to\infty} J_N(\underline \gamma).
\end{equation}
We note that the cost can be made more general: Let for Borel $B \in \mathbb{U}$,
\[\kappa^N(B) = \frac{1}{N}\sum_{p=1}^{{N}} \delta_{U^{i}}(B)\]
be the empirical measure involving control actions, and for Borel $D \in \mathbb{R}^n$
\[\zeta^N(B) = \frac{1}{N}\sum_{p=1}^{{N}} \delta_{X^{i}}(D)\]
be the measure for local states. We could have $c$ above be of the form $c(X^i, U^i,\zeta^N, \kappa^N)$ as well. 

The above will be called either \textbf{true model A} or \textbf{true model B} and be associated with measure $\mu$ under control policy $\underline \gamma_N$. \\

  {\bf A Nominal (Reference) Model with Non-Interacting Agents.} Consider now the nominal system of non-interacting agents given by the collection of stochastic differential equations
\begin{equation*}
\begin{cases}
    dX^1_t&=b_1(X^1_t,U^1_t) dt+\sigma_1(X^1_t) dB^1_t \nonumber \\
    &...\nonumber \\
     dX^N_t&=b_N(X^N_t,U^N_t) dt+\sigma_N(X^N_t) dB^N_t,\label{referenceModel}
     \end{cases}
\end{equation*}
with $X^{[1,N]}_0=x^{[1,N]}_0$. The same assumptions on either of the true models apply here.

The rest of this paper is dedicated to showing a robustness theorem for $J(\underline \gamma)$ where risk-neutral stochastic optimal control with weakly interacting agents may be bounded by risk-sensitive stochastic optimal control with non interacting agents. 

\section{Robust Decentralized Control Design for Coupled Systems via Risk Sensitive Control of Non-Interacting Systems}

We will consider two classes of models, in the first one the coupling will be via mean-field effects and in the second one the coupling will be more general. 

In the former setup, a total energy constraint will be absent unlike the latter setup. In particular, we are interested in proving a robustness bound for $J$ without the Novikov assumption \eqref{eq:Novikov} to be presented below. This will be achieved by utilizing the symmetry in the model, and thus avoid the total energy constraint. Later on, we will consider a more general coupling structure under a more restrictive moment constraint.  

First we present a variational formula involving Cameron-Marton spaces and a formula for relative entropy.

\subsection{A general variational formula and relative entropy}

First, we have the following result for the Cameron-Martin (see e.g. \cite[Chapter 3]{Hairer-Notes}) space of countable product spaces. 
\begin{theorem}\label{theorem:CM-for-product}
    Let $\mu_n, n \in \mathbb{N}$ be centered Gaussian measures on the countable collection of Banach spaces $\mathcal B_n$ with Cameron-Martin spaces $\mathcal H_{\mu_n}$. Consider the Fr\'echet space $\mathcal F =\Pi_{n=1}^\infty \mathcal B_n$ with product topology and product measure $\mu$. Then the Cameron-Martin space is 
    \begin{equation*}
        \mathcal H_\mu=\{\mathbf v=\{v_n\}_{n\in \mathbb N}:\|\mathbf v\|_\mu^2=\sum_{n=1}^\infty \|v_n\|_{\mu_n}^2<\infty\}.
    \end{equation*}
\end{theorem}
\begin{proof}
    The dual space, $\mathcal F^\ast$ can be identified with elements of $\Pi_{n=1}^\infty \mathcal B_n^\ast$ with finite support, with $\mathbf v^\ast \in \Pi_{n=1}^\infty \mathcal B_n^\ast$ being given by

    \[\mathbf v(\mathbf x)=\sum_{n=1}^\infty v_n(x_n)\] 
for $\mathbf x\in \Pi_{n=1}^\infty \mathcal B_n^\ast$. Let $j: \mathcal F^\ast \to L^2(\mathcal F, \mu)$ be the canonical embedding. Then $\mathcal H_\mu$ is the $L^2$ closure of $j(\mathcal F)$ under $\mu$ with norm $\|\mathbf v\|_\mu^2=\|j(\mathbf v)\|_\mu^2=\sum_{n=1}^\infty \|v_n\|_{\mu_n}^2$.
\end{proof}
\begin{remark}\label{remark:1}
    Given any measure on Fr\'echet space $\mathcal F$ there exists a compactly embedded Banach space $\mathcal B$ of full measure. Therefore, we may restrict to that compactly embedded Banach space. See e.g. \cite{Bogachev}, Theorem 3.5. Therefore any results stated in terms of Banach spaces hold in this setting. 
\end{remark}
\begin{theorem}\label{theorem:product-space}
 Let $\mu_n$ be centered Gaussian measures on the countable collection of Banach spaces $\mathcal B_n$ with Cameron-Martin spaces $\mathcal H_{\mu_n}$. Consider the Fr\'echet space $\mathcal F =\Pi_{n=1}^\infty \mathcal B_n$ with product topology, product measure $\mu$ (which may be reduced to a Banach space, given Remark \ref{remark:1}) and Cameron-Martin space  
    \begin{equation*}
        \mathcal H_\mu=\{\mathbf v=\{v_n\}_{n\in \mathbb N}:\|\mathbf v\|_\mu^2=\sum_{n=1}^\infty \|v_n\|_{\mu_n}^2<\infty\}.
    \end{equation*}
Suppose that $\mathbf v$ is a sequence of (potentially random) elements so that the Novikov condition 
\begin{equation}\label{eq:Novikov}
    E_\mu \left[ e^{\frac{1}{2} \|\mathbf v\|_\mu^2 } \right]=E_\mu \left[e^{\frac{1}{2} \sum_{n=1}^\infty \|v_n\|_{\mu_n}^2 }\right]<\infty
\end{equation}
holds. Consider a set of admissible policies $\Gamma=\prod_{n=1}^\infty \Gamma^n$.  Let $C:\mathcal F\times \Gamma \to \mathbb R$ be a map  so that for each $\gamma\in \Gamma$, the functional $ C(\cdot, \gamma)$ satisfies the finite entropy assumption $E_\mu\left[e^{C}|C|\right]<+\infty $. Then we have that 
\begin{align}\label{variationalForm1}
   \log E_\mu \left[e^{C(\cdot,\gamma)}\right]=\sup_{\nu \sim \mu} \left\{E_\nu[C]-D_{KL}(\nu||\mu)\right\},
\end{align}
and the supremum is achieved at $\mu^{\ast,\gamma}$ which has density
\begin{equation*}
    \frac{d\mu^{\ast,\gamma}}{d\mu}=\frac{e^{C(\cdot, \gamma)}}{E_\mu[e^{C(\cdot,\gamma)}]}.
\end{equation*}
\end{theorem}
\begin{proof}
Given the setup, the result follows directly from \cite{BK}. 
\end{proof}

We next present a supporting result on relative entropy.

\begin{lemma}\label{relativeEnt2ndMoment}
Let $v_s^n$ be a sequence of progressively measurable processes so that $E\left[e^{\frac12\sum_{n=1}^\infty \int_0^T (v_s^n)^2 ds}\right]<+\infty$, and let $\mu$ be the law of the shifted process $(B_t^1+\int_0^t v_s^1 ds,...)$, then
\begin{equation}\label{relativeEntBoundR}
D_{KL}(\mu||\mu_0) = \frac{1}{2}\sum_{n=1}^\infty E_\mu \left[\int_0^T (v_s^n)^2 ds.\right]
\end{equation}
\end{lemma}
\begin{proof}
By Girsanov, we know that the density for $\mu$ is $$\frac{d\mu}{d\mu_0}=\exp\left(\sum_{n=1}^\infty \int_0^T v_s^n dB_s^n-\frac{1}{2}\sum_{n=1}^\infty \int_0^T (v_s^n)^2 ds\right).$$ 
Taking $\log$ and expectation with respect to $\mu$ gives
$$D_{KL}(\mu||\mu_0)=E_\mu\left(\sum_{n=1}^\infty \int_0^T v_s^n dB_s^n-\frac{1}{2}\sum_{n=1}^\infty \int_0^T (v_s^n)^2 ds\right).$$
By Girsanov, we have that $B=\tilde B+\int v $ where $\tilde B$ is a Brownian motion under $\mu$. Therefore we conclude that 
$$D_{KL}(\mu||\mu_0)=E_\mu\left(\sum_{n=1}^\infty\left( \int_0^T v_s^n d\tilde B_s^n+\int_0^T (v_s^n)^2 ds \right)-\frac{1}{2}\sum_{n=1}^\infty \int_0^T (v_s^n)^2 ds\right),$$
leading to (\ref{relativeEntBoundR}).

\end{proof}

The above supporting results will be utilized in the following two subsections.

\subsection{Systems with mean-field coupling}

In this section, we consider \textbf{System Model A}. First, we obtain an expression for the relative entropy of path measures on solutions to SDEs. This has been investigated previously in e.g. \cite{Girsanov-path-1,Girsanov-path-2,Girsanov-path-3,Girsanov-path-4}.

\begin{lemma}\label{lem:robust-kld}
    Let $\mu_0^N$ be the law of $\tilde X^{[1,N]}=(\tilde X^1(t),...,\tilde X^N(t))$ and let $\mu^N$ be the law of $X^{[1,N]}=(X^1(t),...,X^N(t))$. 
    Then 
    \begin{align*}
       D_{KL}(\mu^N||\mu_0^N)
       &=\sum_{n=1}^N \frac{1}{2}E_{\mu^N}\left[\int_0^T \left|\frac{v_n(s)}{\sigma_n}\right|^2\right].
    \end{align*}
Suppose that $\sigma_i$ are equal and we have symmetric interaction terms $v_n$.  Then
    \begin{eqnarray}
       \frac{1}{N} D_{KL}(\mu^N||\mu_0^N)=\frac{1}{2}E_{\mu^N}\left[\int_0^T v_1^2(s)ds\right].
    \end{eqnarray}
\end{lemma}
\begin{proof}
    By Girsanov we know that the Radon-Nikodym derivative is
          \begin{equation*}
        \frac{d\mu^N}{d\mu_0^N}=\exp\left(\left(\int_0^T \left(\frac{v_1(s)}{\sigma_1(s)},...,\frac{v_n(s)}{\sigma_N(s)}\right)\cdot dB(s)-\frac{1}{2}\int_0^T \left\|\left(\frac{v_1(s)}{\sigma_1(s)},...,\frac{v_n(s)}{\sigma_n(s)}\right)\right\|^2 ds\right)\right).
    \end{equation*}

In the constant case, the above simplifies as    
    \begin{equation*}
        \frac{d\mu^N}{d\mu_0^N}=\exp\left(\left(\int_0^T \left(\frac{v_1(s)}{\sigma_1},...,\frac{v_n(s)}{\sigma_N}\right)\cdot dB(s)-\frac{1}{2}\int_0^T \left\|\left(\frac{v_1(s)}{\sigma_1},...,\frac{v_n(s)}{\sigma_n}\right)\right\|^2 ds\right)\right).
    \end{equation*}
    Therefore the relative entropy (i.e., the Kullback-Leibler divergence) is 
    \begin{align*}
       D_{KL}(\mu^N||\mu_0^N)&= E_{\mu^N}\left[\left(\int_0^T \left(\frac{v_1(s)}{\sigma_1(s)},...,\frac{v_n(s)}{\sigma_N(s)}\right)\cdot dB(s)-\frac{1}{2}\int_0^T \left\|\left(\frac{v_1(s)}{\sigma_1(s)},...,\frac{v_n(s)}{\sigma_n(s)}\right)\right\|^2 ds\right)\right]\\
       &=\sum_{n=1}^N E_{\mu^N}\left[\int_0^T \frac{v_n(s)}{\sigma_n(s)} dB^n(s)-\frac{1}{2}\int_0^T \left|\frac{v_n(s)}{\sigma_n(s)}\right|^2 ds\right]\\
       &=\sum_{n=1}^N \frac{1}{2}E_{\mu^N}\left[\int_0^T \left|\frac{v_n(s)}{\sigma_n(s)}\right|^2 ds\right],
    \end{align*}
where the last step is because $E_{\mu^N}\left[\int_0^T \frac{v_n(s)}{\sigma_n(s)} dB^n(s)\right]=E_{\mu^N}\left[\int_0^T \left|\frac{v_n(s)}{\sigma_n(s)}\right|^2 ds\right]$.
\end{proof}

We have the following robustness result.

\begin{theorem}\label{theorem:main}
Consider the cost functional $J$ defined in \eqref{eq:J-def} with cost criterion (\ref{generalCostFnMF}) and fixed $\gamma \in \Gamma$. Consider a progressively measurable process $v_t\in L^2([0,T]\times \Omega)$ satisfying Novikov's condition and the measure $\mu_{\gamma}$ given by the law of $\tilde X(t)=(\tilde X^1(t),...)$ and $\mu_{\gamma}^N$ the law of $\tilde X^{[1,N]}=(\tilde X^1(t),...,\tilde X^N(t))$. Let the measure $\mu_{0,\gamma}$ given by the law of $ X(t)=( X^1(t),...)$ and $\mu_{0,\gamma}^N$ the law of $X^{[1,N]}=(X^1(t),...,X^N(t))$. 

%Let $\mu_m^{N,n}$ be the law of the marginal process $\tilde X^n$ and let $\mu_{0,m}^{n}$ be the law of the marginal process $X^n$ (which is independent of $N$). Then,
%     if 
%    \begin{equation}\label{AverageCostLimCond}
%    \limsup_{N\to \infty}\inf_{\underline \gamma\in \Gamma}E_{\mu^N}\left[J_N(\underline \gamma)\right]\geq \inf_{\underline \gamma\in \Gamma} E_\mu[J(\underline \gamma)],
%    \end{equation}
      we have that 
    \begin{align}\label{robustMeanBound2}
     &   \inf_{\underline \gamma\in \Gamma}\limsup_{N\to\infty} E^{\underline \gamma}_\mu[c_N(\omega^{[1,N]},U^{[1,N]})] \nonumber \\
&     \leq  \inf_{\underline \gamma\in \Gamma} \limsup_{N\to\infty} \frac{1}{N} \log E^{\underline \gamma}_{\mu_{0,\gamma}} \left[e^{N c_N(\omega^{[1,N]},U^{[1,N]})}\right] 
        + \limsup_{N\to\infty}  \frac{1}{N}D_{KL}(\mu_{\gamma}^N||\mu_{0,\gamma}^N).
    \end{align}

With identical $E[\int_0^T (v_s^n)^2]$ for each $n$ under symmetric interaction and constant $\sigma_n$ that are bounded away from zero as $|\sigma_n|\geq \sigma>0$, the above leads to
          \begin{align}\label{robustMeanBound333}
    &    \inf_{\underline \gamma\in \Gamma} E^{\underline \gamma}_\mu[c_N(\omega^{[1,N]},U^{[1,N]})]  \nonumber \\
    &\leq  \inf_{\underline \gamma\in \Gamma} \bigg(\limsup_{N\to\infty} \frac{1}{N} \log E^{\underline \gamma}_{\mu_{0,\gamma}} \left[e^{N c_N(\omega^{[1,N]},U^{[1,N]})}\right]+  \frac{1}{2\sigma^2} E^{\underline \gamma}_{\mu} \left[\int_0^T v_s^2ds\right] \bigg).
    \end{align}
    
\end{theorem}

\begin{proof}
Using Donsker-Varadhan (see \cite{BK}), for every policy $\gamma \in \Gamma$, we have the variational expression (suppressing the dependence of the measures on $\gamma$)
\begin{align}
L(\gamma)&:= \limsup_{N\to\infty} \frac{1}{N} \log E^{\underline \gamma}_{\mu_0} \left[e^{N c_N(\omega^{[1,N]},U^{[1,N]})}\right]  \nonumber \\
&= \limsup_{N\to\infty}\sup_{\mu^N\sim \mu_0^N}\left\{ E^{\underline \gamma}_\mu[c_N(\omega^{[1,N]},U^{[1,N]})]  -\frac{1}{N}D_{KL}(\mu^N||\mu_0^N)\right\}.
\end{align}
Therefore, for each fixed sequence $\mu^N \sim \mu_0^N$,
\begin{eqnarray}
L(\gamma) &\geq&  \limsup_{N\to\infty} \left\{ E^{\underline \gamma}_\mu[c_N(\omega^{[1,N]},U^{[1,N]})]  -\frac{1}{N}D_{KL}(\mu^N||\mu_0^N)\right\}. \nonumber \\
&\geq& \limsup_{N\to\infty} \left\{  \left\{ E^{\underline \gamma}_\mu[c_N(\omega^{[1,N]},U^{[1,N]})] \right\} - \limsup_{N\to\infty}  \frac{1}{N}D_{KL}(\mu^N||\mu_0^N) \right\}
\end{eqnarray}
This leads to, for each $\gamma \in \Gamma$
\begin{equation*}
\limsup_{N\to\infty} \left\{  E^{\underline \gamma}_\mu[c_N(\omega^{[1,N]},U^{[1,N]})] \right\} \leq L(\gamma) + \limsup_{N\to\infty}  \frac{1}{N}D_{KL}(\mu^N||\mu_0^N)
\end{equation*}
where we again note that $\mu_{\gamma}$ depends on the policy. As the above applies for each admissible $\gamma$, we have (\ref{robustMeanBound2}).

\end{proof}

Accordingly, one could optimize $L(\gamma)$ and apply the resulting policy to the right hand term of (\ref{robustMeanBound2}) to obtain a further upper bound.

By a simple adaptation of Theorem \ref{theorem:main} we have the corollary for the case of \textbf{System Model B}. 
\begin{corollary}
    Consider the setup of Theorem \ref{theorem:main}. Let $M_0^N$ be the law of $(B^1,...,B^N)$, $M_0$ be the law of $(B^1,...)$, $M^N$ be the law of $(B^1+\int v^1,...,B^N+\int v^N)$ and let $M$ be the law of $(B^1+\int v_1,...)$. Then we have the robustness bound
      \begin{align}\label{robustMeanBound44}
    &    \inf_{\underline \gamma\in \Gamma} \limsup_{N\to\infty} E^{\underline \gamma}_M[c_N(X^{[1,N]},U^{[1,N]})]  \nonumber \\
    &\leq  \inf_{\underline \gamma\in \Gamma} \bigg(\limsup_{N\to\infty} \frac{1}{N} \log E^{\underline \gamma}_{M_0} \left[e^{N c_N(X^{[1,N]},U^{[1,N]})}\right]+  \frac{1}{2} E^{\underline \gamma}_{M} \left[\int_0^T v_s^2ds\right] \bigg).
    \end{align}
\end{corollary}
\begin{proof}
    Apply Theorem \ref{theorem:main} with $\tilde v_t^i=\sigma_i(X^i)v_t^i$.
\end{proof}

\subsection{Systems with arbitrary coupling: Measures on Infinite product space and total energy constraint on coupling}\label{sec:Gaussian}

In this section, under the general cost criterion (\ref{generalCostFnD}), we consider the setup where the noise of the true system is a perturbation of the noise of the nominal system directly on the infinite product space arbitrarily, though under a finite total energy condition. In particular this finite energy condition prohibits symmetric interaction. However, this is at the benefit of a general cost criterion. 

Building on the analysis above, and Lemma \ref{relativeEnt2ndMoment}, we obtain the following.

\begin{theorem}\label{theorem:main-Gaussian-measure}
    Let $J$ be the functional defined in equation \eqref{eq:J-def} under cost criterion (\ref{generalCostFnD}). Then for fixed $\gamma\in \Gamma$ and for $\mu\sim \mu_0$ ($\mu$ corresponding to \textbf{System Model B}) where $\mu_0$ is the law of the infinite collection of i.i.d. Brownian motions, we have that
    \begin{align*}
        E^{\underline \gamma}_\mu[c_N(X^{[1,N]},U^{[1,N]})]&\leq \log E^{\underline \gamma}_{\mu_0} \left[e^{c_N(X^{[1,N]},U^{[1,N]})}\right] + D_{KL}(\mu||\mu_0)\\
        &=\log E^{\underline \gamma}_{\mu_0} \left[e^{c_N(X^{[1,N]},U^{[1,N]})}\right] + \frac{1}{2}\sum_{n=1}^\infty E^{\underline \gamma}_\mu \left[\int_0^T (v_s^n)^2 ds\right],
    \end{align*}
    implying, via Lemma \ref{relativeEnt2ndMoment}, that
    \begin{equation*}
     \inf_{\underline \gamma\in \Gamma}    E^{\underline \gamma}_\mu[c_N(X^{[1,N]},U^{[1,N]})] \leq \inf_{\underline \gamma\in \Gamma}\log E^{\underline \gamma}_{\mu_0} \left[e^{c_N(X^{[1,N]},U^{[1,N]})}\right]  +  \frac{1}{2}\sum_{n=1}^\infty E^{\underline \gamma}_{\mu} \left[\int_0^T (v_s^n)^2 ds\right].
    \end{equation*}
\end{theorem}
\begin{comment}
{\bf Implication.} We thus conclude that we can obtain an upper bound on the true cost function for the true system, by studying
\[\inf_{\underline \gamma\in \Gamma}\log E_{\mu_0} \left[e^{c_N(X^{[1,N]},U^{[1,N]})} \right] \]
for the reference model, which is an uncoupled model where the particles are only coupled via their cost (and not their dynamics). In particular, when $C$ is a quadratic cost functional and the reference model dynamics are linear, this is known to be a convex problem. 

%\sy{We note that the results in this section are %applicable even when the cost function is %dependent on the environment not necessarily via %the mean-field effects. TO ZACK: Should this be %emphasized?}

\sy{Move the later to elsewhere or provide better interpretation.}
\end{comment}

\section{Refinements: Optimizing Sensitivity, Robustness via Symmetric Models, Stability to Small Couplings and Large Mean-Field Problems}\label{sec:optimal-risk-sensitivity}

In this section, we obtain several refinements and implications, including on optimizing sensitivity parameter in the risk sensitive formulation, and further structural and quantitative properties.

\subsection{Optimizing sensitivity}

Given $\alpha \in \mathbb R$ and functional $C$ we can consider the expression
    \begin{equation*}
       F(\alpha, \underline \gamma):= \frac{1}{\alpha}\log E_{\mu_0}\left[\exp\left(\alpha C(\underline\gamma)\right)\right].
    \end{equation*}
In light of the following proposition, $\alpha$ can be seen to represent the ``amount" of risk sensitivity. 
\begin{proposition}[\cite{Limit-of-alpha}, Proposition 3.1]
If $C$ is essentially bounded, we have that 
\begin{align*}
    \lim_{\alpha\to \infty} F(\alpha, \underline \gamma)&=\operatorname{esssup}[C(\underline \gamma)]\\
    \lim_{\alpha\to 0} F(\alpha, \underline \gamma)&=E_{\mu_0}[C(\underline \gamma)]\\
    \lim_{\alpha\to -\infty} F(\alpha, \underline \gamma)&=\operatorname{essinf}[C(\underline \gamma)],
\end{align*}
\end{proposition}
The limits representing the infinite risk-averse (where only the worst case matters), risk-neutral and infinite risk-seeking (where only the best case matters) cases, respectively. In this sense, $\alpha$ parameterizes the ``amount" of risk sensitivity. If we apply Donsker-Varadhan to $F$, we have
\begin{equation*}
     F(\alpha, \underline \gamma)=\sup_{\mu\sim \mu_0}\left\{E_{\mu}[C(\underline\gamma)]-\frac{1}{\alpha}D_{KL}(\mu||\mu_0)\right\},
\end{equation*}
and for fixed $\mu\sim \mu_0$ we have the robustness bound 
\begin{equation}\label{eq:bound-with-alpha-1}
    E_{\mu}[C(\underline\gamma)]\leq F(\alpha, \underline \gamma)+\frac{1}{\alpha} D_{KL}(\mu||\mu_0)=:\psi(\alpha, \underline \gamma),
\end{equation}
giving an extra degree of freedom $\alpha$. The following proposition gives a qualitative expression for the optimal $\alpha$. 

\begin{theorem}\label{thmoptimizeparam}
For fixed $\underline \gamma$ if $\alpha^\ast$ is a minimizer of $\psi(\alpha, \underline \gamma)$ the bound \eqref{eq:bound-with-alpha-1} reduces to 
\begin{equation*}
     E_{\mu}[C]\leq E_{\hat \mu}[C],
\end{equation*}
where
\[\frac{d\hat \mu}{d\mu_0}=\frac{1}{E_{\mu_0}[e^{\alpha^\ast C}]}e^{\alpha^\ast C}\]
is the solution to the variational problem
\[\sup_{\nu\sim \mu_0}\left\{E_{\nu}[C(\underline\gamma)]-\frac{1}{\alpha}D_{KL}(\nu||\mu_0)\right\}.\]
Additionally, we have that 
\[D_{KL}(\hat \mu||\mu_0)=D_{KL}(\mu||\mu_0).\]
\end{theorem}
\begin{proof}
$\psi$ is differentiable in $\alpha$, with derivative
\begin{align*}
    \psi_\alpha(\alpha, \underline\gamma)&=\frac{1}{\alpha} \frac{1}{E_{\mu_0}[e^{\alpha C}]}E_{\mu_0}[e^{\alpha C} C]-\frac{1}{\alpha^2}\log E_{\mu_0}[e^{\alpha C}]-\frac{1}{\alpha^2}D_{KL}(\mu||\mu_0)\\
    &=\frac{1}{\alpha}\left(E_{\hat \mu }[C]-\psi(\alpha, \underline \gamma)\right),
\end{align*}
where
\[\frac{d\hat \mu}{d\mu_0}=\frac{1}{E_{\mu_0}[e^{\alpha^\ast C}]}e^{\alpha^\ast C}\]
is the solution to the variational problem
\[\sup_{\nu\sim \mu_0}\left\{E_{\nu}[C(\underline\gamma)]-\frac{1}{\alpha}D_{KL}(\nu||\mu_0)\right\}.\]
If $\alpha^\ast=\alpha^\ast(\underline \gamma)$ is a minima it satisfies  
\begin{equation*}
    \psi_\alpha(\alpha^\ast, \underline\gamma)=\frac{1}{\alpha^\ast}\left(E_{\hat \mu }[C]-\psi(\alpha^\ast, \underline \gamma)\right)=0.
\end{equation*}
In this case, the robustness bound \eqref{eq:bound-with-alpha-1} becomes 
\begin{equation*}
     E_{\mu}[C]\leq E_{\hat \mu}[C].
\end{equation*}
Additionally, plugging in $\alpha^\ast$ to $\psi$ gives 
\begin{align*}
    \psi(\alpha^\ast, \underline\gamma)&= \sup_{\nu\sim \mu_0}\left\{E_{\nu}[C(\underline\gamma)]-\frac{1}{\alpha^\ast}D_{KL}(\nu||\mu_0)\right\}+\frac{1}{\alpha^\ast}D_{KL}(\mu||\mu_0)\\
    &=E_{\hat \mu}[C(\underline\gamma)]-\frac{1}{\alpha^\ast}D_{KL}(\hat \mu||\mu_0)+\frac{1}{\alpha^\ast}D_{KL}(\mu||\mu_0)\\
    &= \psi(\alpha^\ast, \underline\gamma)-\frac{1}{\alpha^\ast}D_{KL}(\hat \mu||\mu_0)+\frac{1}{\alpha^\ast}D_{KL}(\mu||\mu_0),
\end{align*}
implying that 
\[D_{KL}(\hat \mu||\mu_0)=D_{KL}(\mu||\mu_0).\]
\end{proof}

\subsection{Risk-sensitive control of non-symmetrically interacting particles with one that is symmetrically coupled}
As a corollary of the above, we have the following.
\begin{corollary}
    Let $\mu$ correspond to a system model A or B with non-symmetric interaction terms $v_t^i$ but with a symmetric mean-field cost $C$. Then we can bound the non-symmetric problem $\inf_{\underline\gamma}E_{\mu}[C(\underline\gamma)]$ by a symmetric problem $\inf_{\underline\gamma}E_{\hat \mu}[C(\underline\gamma)]$, where $\frac{d\hat \mu}{d\mu_0}\propto e^{\alpha^\ast C(\underline \gamma)}$ for some $\alpha^\ast\in \mathbb R$. 
\end{corollary}

The benefit of such a result is that the classical mean-field theory is now applicable as both the interaction terms and the cost terms are now symmetric and coupled via mean-field under the nominal model; this would not have been possible if the dynamics were not symmetrically coupled.

\subsection{Stability of risk-sensitive solutions to small interactions}\label{stabilitySectionPerturb}
Optimal risk-sensitivity can lead to convergence of the robustness bound \eqref{eq:bound-with-alpha-1} to an equality. By choosing $\alpha$ appropriately, we can recover the risk-neutral problem in the small interaction limit. 
\begin{proposition}
    Consider the setup of Theorem \ref{theorem:main} or \ref{theorem:main-Gaussian-measure} with $v_{t,n}^i=\varepsilon_n u_t^i$ and $\varepsilon_n\to 0$. Then there is a choice of $\alpha_n$ so that relation \eqref{eq:bound-with-alpha-1} is an equality in the limit as $n\to \infty$.
\end{proposition}
\begin{proof}
    Using Lemma \ref{lem:robust-kld}, \eqref{eq:bound-with-alpha-1} becomes 
    \begin{equation}\label{eq:alpha-to-zero}
         E_{\mu^{\varepsilon_n}}[C(\underline\gamma)]\leq F(\alpha, \underline \gamma)+\frac{\varepsilon_n^2}{\alpha} D_{KL}(\mu^1||\mu_0),
    \end{equation}
    where $\mu^{\varepsilon}$ is the measure corresponding to system model A or B with $v_t^i=\varepsilon u_t^i$ and $\mu^1$ is the law of system model A or B with $\varepsilon =1$. In this case, we may let $\alpha_n$ be so that $\lim_{n\to\infty} \frac{\varepsilon_n^2}{\alpha_n}=0$ and $\lim_{n\to\infty} \alpha_n=0$. Taking the limit $n\to\infty$ in equation \eqref{eq:alpha-to-zero} concludes the proof. 
\end{proof}

\subsection{The mean-field limit as the nominal model}

For a mean-field model, in the limit of large $N$, if we have that the agents apply symmetric policies, then it is known that the strategic measure converges, in some sense, to a collection of independent product measures \cite{Girsanov-path-1,Girsanov-path-2,Girsanov-path-3,Girsanov-path-4}. Accordingly, it may be more appropriate, for some applications, to consider such a mean-field limit as the reference model. Here, the robustness can be regarded with regard to an uncertainty on the number of agents in the model but with the information that the number is not small. 

Notably, consider the model:
\begin{equation*}
    d\tilde X^i_t=(b(\tilde X^i_t,\zeta^N_t, U^i_t) +\sigma(\tilde X^i_t)v_t^1) dt+\sigma(\tilde X^i_t) dB^i_t, \quad i=1,\cdots,N \label{trueModel-MFB}
\end{equation*}
where we assume that $X^i_0$ are i.i.d. random variables with finite second moments. 
We write the above as:
\begin{equation*}
    d\tilde X^i_t=b(\tilde X^i_t, \mu_t, U^i_t)dt + \bigg(b(\tilde X^i_t,\zeta^N_t, U^i_t) - b(\tilde X^i_t,\mu_t, U^i_t) \bigg) dt+\sigma(\tilde X^i_t) dB^i_t , \quad i=1,\cdots,N \label{trueModel-BMFB}
\end{equation*}
which is now to be studied as an instance of the model considered earlier in (\ref{trueModel-B}) with
\[v_t^i = \bigg(b(\tilde X^i_t,\zeta^N_t, U^i_t) - b(\tilde X^i_t,\mu_t, U^i_t) \bigg),\]
so that we have the decoupled model
\begin{equation}
    d X^i_t=b( X^i_t, \mu_t, U^i_t)dt +\sigma( X^i_t) dB^i_t,  \quad i=1,\cdots,N \label{decoupMFL}
\end{equation}
Note that, here if $\mu_t$ is the probability measure defined by $X^i_t$ under an identical Markov policy applied by each agent, then (\ref{decoupMFL}) would reduce to what is known as the McKean-Vlasov dynamics. Thus, models of the type above are known to be related to solutions of centralized McKean-Vlasov control problems under an apriori symmetry in control policies \cite{lacker2017limit}). It is known that $\zeta^N$ concentrates around $\mu$ as $N$ increases with explicit convergence rates.  

Thus, the deviation in the empirical mean from the asymptotic limit of the empirical mean serves as the coupling term. In view of this, our risk-sensitivity bound writes as, by adapting (\ref{robustMeanBound333}) and (\ref{eq:bound-with-alpha-1}), for any $\alpha > 0$ and any policy $\underline \gamma$:

\begin{equation}\label{RobustMeanBound333111}
\begin{split}
    &    \limsup_{N\to\infty} E^{\underline \gamma}_\mu[c_N(\omega^{[1,N]},U^{[1,N]})]   \\
    &\leq  \limsup_{N\to\infty} \frac{1}{N}  \frac{1}{\alpha} \log E^{\underline \gamma}_{\mu_0} \left[e^{\alpha N c_N(\omega^{[1,N]},U^{[1,N]})}\right]+  \frac{1}{2\sigma^2}\frac{1}{\alpha}  E^{\underline \gamma}_{\mu} \left[\int_0^T (v^1_s)^2ds\right].
    \end{split}
    \end{equation}

We will focus on the first term on the right hand side in equation \eqref{RobustMeanBound333111} and vary $\alpha$. We note that $\mu_t$ will also change depending on the parameters as the control policies will change.

Now, consider
\[ \lim_{\alpha \to 0} \bigg( \limsup_{N\to\infty} \frac{1}{N}  \frac{1}{\alpha} \log E^{\underline \gamma}_{\mu_0} \left[e^{\alpha N c_N(\omega^{[1,N]},U^{[1,N]})}\right] \bigg),\]
where the limit is well-defined as the expression is monotonically increasing in $\alpha$. By interchanging limit and $\inf$ we get the lower bound
\begin{align}
&\inf_{\underline \gamma\in \Gamma} \lim_{\alpha \to 0} \bigg( \limsup_{N\to\infty} \frac{1}{N}  \frac{1}{\alpha} \log E^{\underline \gamma}_{\mu_0} \left[e^{\alpha N c_N(\omega^{[1,N]},U^{[1,N]})}\right] \bigg) \nonumber \\
& \geq \limsup_{\alpha \to 0} \inf_{\underline \gamma\in \Gamma} \bigg( \limsup_{N\to\infty} \frac{1}{N}  \frac{1}{\alpha} \log E^{\underline \gamma}_{\mu_0} \left[e^{\alpha N c_N(\omega^{[1,N]},U^{[1,N]})}\right] \bigg).
\end{align}
If we have a policy sequence $\underline \gamma_n$ which attains the infimum for a sequence of $\alpha_n \downarrow 0$, this implies that
\[\limsup_{\alpha_n \to 0} \inf_{\underline \gamma\in \Gamma} \bigg( \limsup_{N\to\infty} \frac{1}{N}  \frac{1}{\alpha_n} \log E^{\underline \gamma}_{\mu_0} \left[e^{\alpha_n N c_N(\omega^{[1,N]},U^{[1,N]})}\right] \bigg) \]   
may provide a bound on the optimal cost provided that the relative entropy term in \eqref{RobustMeanBound333111} can be cautiously studied via a measure concentration analysis as $N \to \infty$. 

Let
\[J_{\alpha}(\underline \gamma) := \limsup_{N\to\infty} \frac{1}{N}  \frac{1}{\alpha} \log E^{\underline \gamma}_{\mu_0} \left[e^{\alpha N c_N(\omega^{[1,N]},U^{[1,N]})}\right] \]
and
\[\inf_{\underline \gamma} J_{\alpha^n}(\underline \gamma) = J_{\alpha^n}(\underline \gamma^n),\]
be so that $\underline \gamma^n$ is optimal for $\alpha_n$. 

We will consider a problem in which the systems are only coupled through the dynamics and not via the cost. Accordingly, we consider instead of (\ref{generalCostFnMF}) the cost function
\begin{equation}\label{generalCostFnMF2}
c_N(X^{[1,N]},U^{[1,N]}) = \frac{1}{N} \sum_{i=1}^N c\left(X^i, U^i\right)
\end{equation}

For such problems, for any $\alpha_n$, since the dynamics are decoupled, an optimal policy for $J_{\alpha}(\underline \gamma)$ will be Markovian (see e.g. Section \ref{principEigenEx}). This will be true even if we allow for policies that are centralized, that is $\sigma(X^1_{[0,t]},\cdots,X^N_{[0,t]})$-measurable. We also assume the following.
\begin{assumption}\label{LipBoundB1}
Uniform over $x^i,u^i$, we have
\[|b(x^i,\zeta^N, u^i) - b(x^i,\mu, u^i)| \leq C W_1(\zeta^N_t,\mu_t),\]
where $W_1$ is the Wasserstein metric of order $1$ and $C < \infty$. 
\end{assumption}

\begin{proposition}
    Let $U^i$ be Markovian policies only depending on $X^i$ and let $\sigma$ be constant. Let $b$ be so that 
    \begin{equation}
        \sup_{t\geq 0, x\neq y, m\in \mathcal P(\mathbb R),\underline \gamma \in \Gamma}\left\{\frac{|b(t,x,U^i,m)-b(t,y,U^i,m)|}{|x-y|}\right\}<\infty,
    \end{equation}
    and assume Assumption \ref{LipBoundB1}. Assume that $b$ is twice differentiable with respect to $m$ with Lipschitz derivatives uniform in the other variables. Also assume that $\mu$ has greater than $4$ moments. 
Then we have that 
\begin{equation}\label{concentBound1}
    \frac{1}{2\sigma^2}\frac{1}{\alpha_n}  E^{\underline \gamma}_{\mu} \left[\int_0^T (v^1_s)^2ds\right]\leq K \frac{1}{2\sqrt{N}\alpha_n\sigma^2},
\end{equation}
where $K$ is independent of $U$ and $N$.
\end{proposition}
\begin{proof}

    Define $b_U(t,x,m)=b(t,x,U,m)$. Then we have by the assumptions on $b$ and \cite[Theorem 3.8]{de2021backward} that 
    \begin{align*}
         \frac{1}{2\sigma^2}\frac{1}{\alpha_n}  E^{\underline \gamma}_{\mu} \left[\int_0^T (v^1_s)^2ds\right]&\leq \frac{1}{2\sigma^2}\frac{1}{\alpha_n}  E^{\underline \gamma}_{\mu} \left[\int_0^T (W_1(\zeta_t^N, \mu_t))^2ds\right]\\
         &\leq \frac{C}{2\sigma^2}\frac{1}{\alpha_n}  E^{\underline \gamma}_{\mu} \left[\int_0^T (W_2(\zeta_t^N, \mu_t))^2ds\right]\\
         &\leq \frac{CK_U^+ T}{2\sigma^2}\frac{1}{\alpha_n\sqrt{N}},
    \end{align*}
where $K_U^+$ is independent of $N$. By Jensen's inequality, $W_2 \geq W_1$ in the above. By the assumptions on $b_U$, we have that $K:=\sup_U C TK_U^+<\infty$. 
\end{proof}

Accordingly, for any  symmetric policy, $\alpha_n$ (depending on $N$), can be designed so that
\[ \frac{1}{2\sigma^2}\frac{1}{\alpha_n}  E^{\underline \gamma}_{\mu} \left[\int_0^T (v^1_s)^2ds\right] \to 0,\]
due to measure concentration involving copies of i.i.d. stochastic flows.

Thus, if we can show that $\underline \gamma_n$ has a converging subsequence and the cost is lower semi-continuous, then we will be able to establish the optimality of symmetric policies. Now, we have that on the space of Markov policies the map $\underline \gamma \mapsto X_{[0,T]}$ is weakly continuous when the policy space is endowed with the {\it Borkar topology} \cite{borkar1989topology,pradhan2022near} and this applies also for resulting McKean-Vlasov model given in (\ref{decoupMFL}) where $\mu_t$ is the law of $X_t$ under Assumption \ref{LipBoundB1} by \cite[Theorem 4.18]{carmona2018probabilistic} (note that by Jensen's inequality $W_2 \geq W_1$ thus \cite[Theorem 4.18]{carmona2018probabilistic} is applicable under Assumption \ref{LipBoundB1}). 
 
Accordingly, for any $\alpha$ the cost is lower semi-continuous in the control.  Furthermore, the space of policies is compact under this topology \cite{borkar1989topology}. Accordingly, 
\begin{eqnarray}\label{inequalitiesBound}
&&\liminf_{\alpha_n \downarrow 0} J_{\alpha_n}(\underline \gamma^n) \nonumber \\
&&\geq \liminf_{\alpha_n \downarrow 0} J_{0}(\underline \gamma^n) \nonumber \\
&&\geq J_{0}(\underline \gamma^{\infty})
\end{eqnarray}
for any limit point $\underline \gamma^{\infty}$ of the sequence $\underline \gamma^n$. The first inequality follows since the cost is increasing in the risk-parameter $\alpha$ for a given policy. This follows from the fact that if we view 

\begin{eqnarray}
        F(\alpha, \underline \gamma)&:=& \frac{1}{\alpha}\log E_{\mu_0}\left[\exp\left(\alpha C(\underline\gamma)\right)\right]. \nonumber \\
&&       = \log \bigg(E_{\mu_0}\left[\bigg(\exp\left( C(\underline\gamma)\right)\bigg)^{\alpha}\right] \bigg)^{\frac{1}{\alpha}},\nonumber 
    \end{eqnarray}
and observing that the $p$-norm of a random variable is non-decreasing in $p$ for $p > 0$, which follows from Jensen's inequality. 

The second inequality is due to lower semi-continuity in the control policies (since the path measure is weakly continuous in the policy as discussed earlier, and the cost is lower semi-continuous in the path measure).

Furthermore, since $\underline \gamma^{\infty}$ is symmetric, the inequalities in (\ref{inequalitiesBound}) are all equalities since we have:

\begin{eqnarray}\label{inequalitiesBound2}
&& J_{0}(\underline \gamma^{\infty}) \nonumber \\
&&=\lim_{\alpha_n \downarrow 0} J_{\alpha_n}(\underline \gamma^{\infty}) \nonumber \\
&&= \liminf_{\alpha_n \downarrow 0} J_{\alpha_n}(\underline \gamma^{\infty}) \nonumber \\
&& \geq \liminf_{\alpha_n \downarrow 0} J_{\alpha_n}(\underline \gamma^n) \nonumber \\
&&\geq \liminf_{\alpha_n \downarrow 0} J_{0}(\underline \gamma^n) \nonumber \\
&&\geq J_{0}(\underline \gamma^{\infty})\nonumber
\end{eqnarray}

Accordingly, given (\ref{concentBound1}), an agent-wise weak limit of $\underline \gamma_n$, $\underline \gamma^{\infty}$, will satisfy the relation:
\begin{align}\label{robustMeanBound333_111}
    & \inf_{\underline \gamma}   \limsup_{N\to\infty} E^{\underline \gamma}_\mu[c_N(\omega^{[1,N]},U^{[1,N]})]  \leq  J_{0}(\underline \gamma^{\infty}).
    \end{align}

Finally, observe that

\begin{align}
J_{0}(\underline \gamma^{\infty})
&=\lim_{\alpha \to 0} \limsup_{N\to\infty} \frac{1}{N}  \frac{1}{\alpha} \log E^{\underline \gamma^{\infty}}_{\mu_0} \left[e^{\alpha N c_N(\omega^{[1,N]},U^{[1,N]})}\right] \nonumber \\
&=\limsup_{N\to\infty} \frac{1}{N}    E^{\underline \gamma^{\infty}}_{\mu_0} \left[c_N(\omega^{[1,N]},U^{[1,N]})\right]
=\limsup_{N\to\infty} \frac{1}{N}   E^{\underline \gamma^{\infty}}_{\mu} \left[c_N(\omega^{[1,N]},U^{[1,N]})\right]
\end{align}
since in the limit given a symmetric policy, the attained expected cost values under the true model and the reference model, are identical. % [EXPAND THE LAST EQUALITY, INVOKING A CONTINUITY RESULT ON THE EMPIRICAL MEASURE]. 

\begin{theorem}
Consider the model (\ref{trueModel-MFB}) and cost function (\ref{generalCostFnMF2}), with Assumption \ref{LipBoundB1}. In the limit of large $N$, a decentralized policy which only uses local information, and identical across all agents, solving (\ref{decoupMFL}) optimizes the upper bound (\ref{RobustMeanBound333111}) (among all admissible, possibly centralized, policies). 
%Furthermore, as $\alpha_n \to 0$, any weak limit of such symmetric identical policies solves \eqref{eq:J-def}
%with (\ref{generalCostFnMF2}) among all centralized or decentralized policies. 
%For large but finite $N$, an upper bound on optimal cost can be obtained by solving a robust risk-sensitive cost problem which solves the reference model. 
\end{theorem}

The result presented above is related to and consistent with the recent results in \cite{sanjarisaldiyuksel2023CTMF}, with a robustness angle for finite numbers of agents: In both discrete-time \cite{sanjari2021optimality,sanjari2020optimality} and continuous-time \cite{lacker2017limit,sanjarisaldiyuksel2023CTMF}, it has been shown that for problems of the type with (\ref{trueModel-MFB}) and cost function (\ref{generalCostFnMF2}), optimal policies are symmetric and independent. Related risk-sensitive problems have been considered in \cite{bensoussan2017risk,moon2016linear,saldi2020approximate,saldi2022partially,tembine2013risk} in continuous-time.  

%Such an approach entails an interesting optimization problem: If all agents apply a robust criterion, the mean-field limit will correspond to collective action of such agents determining the value $\mu$ in the decoupled model (\ref{decoupMFL}). Here $\mu$ can be obtained directly by solving a risk-sensitive McKean-Vlasov problem since the solution of the decoupled agents will be symmetric following the analysis in Section \ref{principEigenEx}. For solutions of risk-sensitive mean-field problems of the type noted above, please see \cite{bensoussan2017risk,moon2016linear,tembine2013risk}. 

\subsection{The discrete-time setup}\label{sec:witsenInfoStructureReview}

We note that the same approach can be applied to the discrere-time setup. For the discrete-time setup, the uncoupled dynamics also exhibit some structural properties which allow for an easier computation of robust solutions. Several reduction results to decoupled teams have been studied, see e.g. \cite[Chapter 3]{YukselBasarBook24}, where a coupled stochastic team is transformed into a decoupled one by change of measure arguments generalizing \cite{wit88}.

In particular, an uncoupled risk-sensitive problem may be convex \cite{YukselSaldiSICON17} whereas the coupled model is not. Under convexity and exchangeability, one can show that optimal policies are symmetric \cite{sanjari2021optimality}. 

%\begin{definition} (Symmetrically optimal teams)\label{Def:ost}\\
%A team is \textit{symmetrically optimal}, if for each given policy $\underline{\gamma}_{T}=(\underline{\gamma}_{T}^{1},\dots,\underline{\gamma}_{T}^{N})$, there exists an identically symmetric policy (i.e., each DM has the same policy, $\underline{\tilde{\gamma}}_{T}=(\underline{\tilde{\gamma}}_{T}^{1},\dots,\underline{\tilde{\gamma}}_{T}^{N})$ , and $\underline{\tilde{\gamma}}_{T}^{i}=\underline{\tilde{\gamma}}_{T}^{j}$ for all $i,j=1,\dots,N$,) which performs at least as good as the given policy.
%\end{definition}
%
%\begin{lemma}\label{lem:mflem}\cite{sanjari2019optimal}
%Consider a team problem defined as ($\mathcal{P}_{T}^{N,\text{MF}}$) with a symmetric information structure and assume that it is convex in policies. Let the action space of each DM be compact and convex. If the cost function is exchangeable, the team is symmetrically optimal.
%\end{lemma}
%
%Accordingly, if we have a decentralized stochastic control problem, under an exponential cost criterion, the problem will be convex. In the absence of convexity, if we only have exchangeability, it can be shown that symmetric policies are optimal for Problem ($\mathcal{P}_{T}^{\infty}$) \cite{sanjari2021optimality}. 

\section{Example: Infinite Horizon Average Cost Criterion and Robustness via an HJB equation}\label{principEigenEx}

The analysis in the paper motivates the study of the risk-sensitive criterion:
\[\inf_{\underline \gamma\in \Gamma} \limsup_{N\to\infty} \frac{1}{N} \log E^{\underline \gamma}_{\mu_0} \left[e^{N c_N(\omega^{[1,N]},U^{[1,N]})}\right].\]

We will study an explicit example in the following. Consider the model given in System Model A with identical particles in (\ref{trueModel}). Suppose that the goal is to minimize the mean-field cost:
\begin{equation*}
    J_N(\underline \gamma_N)=\limsup_{T \to \infty} \frac{1}{T} E^{\underline \gamma_N}_{\mu_0^{N,T}} \left[c^T_N(X^{[1,N]},U^{[1,N]})\right],
\end{equation*}
with
\begin{equation}\label{generalCostFnMFDec}
c^T_N(X^{[1,N]},U^{[1,N]}) = \frac{1}{N} \sum_{i=1}^N c^T\left(X^i, U^i\right),
\end{equation}

where
\[c^T\left(X^i, U^i \right) = \int_0^{T} g(X^i_s, U^i_s) ds \]

where $\underline \gamma_N=(\gamma^1,...,\gamma^N)$ with $U^i_t=\gamma^i(X^i_{[0,t]})$.
We are also interested in the limit
\begin{equation}\label{eq:J-defDec}
    J(\underline \gamma)=\limsup_{N\to\infty} J_N(\underline \gamma).
\end{equation}

Note that here we consider a system where there is coupling in the dynamics but not in the cost.

Following steps similar to that in the proof of Theorem \ref{theorem:main}, consider the expression
\begin{equation}\label{RobustnessExpTermRef2}
L:=\inf_{\underline \gamma\in \Gamma} \limsup_{T\to\infty} \frac{1}{T} \frac{1}{N} \log E^{\underline \gamma}_{\mu_0^{N,T}} \left[e^{N c^T_N(\omega^{[1,N]},U^{[1,N]})}\right].
\end{equation}
Using Donsker-Varadhan (see \cite{BK}), we have the variational expression
\begin{equation*}
L= \inf_{\underline \gamma\in \Gamma}\limsup_{T \to\infty} \frac{1}{T} \frac{1}{N} \sup_{\mu^N\sim \mu_0^N}\left\{ E^{\underline \gamma}_\mu[c_N(\omega^{[1,N]},U^{[1,N]})]  -\frac{1}{N}D_{KL}(\mu^N||\mu_0^N)\right\}.
\end{equation*}
For each fixed $\gamma \in \Gamma$ and $\mu^{N,T} \sim \mu_0^{N,T}$,
\begin{eqnarray}
&&\limsup_{T\to\infty} \frac{1}{T} \frac{1}{N} \log E^{\underline \gamma}_{\mu_0^{N,T}} \left[e^{N c^T_N(\omega^{[1,N]},U^{[1,N]})}\right] \nonumber\\
&&\geq  \limsup_{T \to\infty} \frac{1}{T} \frac{1}{N} \left\{ E^{\underline \gamma}_\mu[c^T_N(\omega^{[1,N]},U^{[1,N]})]  - D_{KL}(\mu^{N,T}||\mu_0^{N,T})\right\}. \nonumber \\
&&\geq \limsup_{T\to\infty} \frac{1}{T} \frac{1}{N}   \left\{ E^{\underline \gamma}_\mu[c^T_N(\omega^{[1,N]},U^{[1,N]})] \right\} - \limsup_{T\to\infty} \frac{1}{T} \left\{\frac{1}{N}D_{KL}(\mu^{N,T}||\mu_0^{N,T}) \right\}.
\end{eqnarray}
This leads to, for each $\gamma \in \Gamma$
\begin{equation*}
\limsup_{T\to\infty} \frac{1}{T} \frac{1}{N}\left\{  E^{\underline \gamma}_\mu[c^T_N(\omega^{[1,N]},U^{[1,N]})] \right\} \leq \limsup_{T\to\infty} \frac{1}{T} \frac{1}{N} \log E^{\underline \gamma}_{\mu_0^{N,T}} \left[e^{N c^T_N(\omega^{[1,N]},U^{[1,N]})}\right] + \limsup_{T\to\infty} \frac{1}{T} \frac{1}{N}D_{KL}(\mu^{N,T}||\mu_0^{N,T})
\end{equation*}
where we again note that $\mu$ depends on the policy. As the above applies for each admissible $\gamma$, it implies that
\begin{equation*}
\inf_{\underline \gamma\in \Gamma}\limsup_{T\to\infty} \frac{1}{T} \frac{1}{N} \left\{  E^{\underline \gamma}_\mu[c^T_N(\omega^{[1,N]},U^{[1,N]})] \right\} \leq L + \limsup_{T\to\infty} \frac{1}{T} \frac{1}{N}D_{KL}(\mu_{\gamma^*}^{N,T}||\mu_{0,\gamma^*}^{N,T}),
\end{equation*}
where $\gamma^*$ is a policy attaining $L$.

For this system, due to the independence of the measures under $\mu_0^{N,T}$, the risk-sensitive criterion 
\[ \limsup_{T \to \infty} \frac{1}{T} \frac{1}{N}\log\bigg( E^{\gamma}[e^{\sum_{i=1}^N\int_0^T c(X^i_s,U^i_s)ds}] \bigg)\]
reduces to the minimization (via the properties of the logarithm) of
\[ \limsup_{T \to \infty} \frac{1}{T}\log\bigg( E^{\gamma}[e^{\int_0^T c(X^i_s,U^i_s)ds}] \bigg),\]
for each of the decoupled particles.
The solution to this minimization, under growth and regularity conditions on $b$ and $\sigma$, has been reported in \cite[Theorems 4.1 and 4.2]{arapostathis2019strict}, and is obtained via the following HJB equation:
\begin{align}\label{HJBRSK}
 \min_{u \in \mathbb{U}} \bigg\{ {\cal A}\psi(x,u) + c(x,u)\psi(x) \bigg\} = \lambda^* \psi(x), \quad \psi(0)=1, 
 \end{align}
where ${\cal A}$ is the generator of the controlled diffusion given in {\bf System Model A} (with identical particle dynamics) and $\lambda^*$ is the optimal cost (known as the principal eigenvalue of the HJB equation above). Let $\gamma^*$ solve the HJB equation. In particular, $\gamma^*$ is a stationary Markov policy.

Accordingly, the robustness bound obtained leads to the following.
\begin{theorem}
Consider the coupled model (\ref{trueModel}) and cost criterion \ref{eq:J-defDec}. The following holds
\[\inf_{\gamma} J(\underline \gamma) \leq \lambda^* + \frac{1}{2\sigma^2}  \lim_{T \to \infty} \frac{1}{T} E^{\underline \gamma^*}_{\mu}\left[\int_0^T v_s^2ds\right],\]
where $\lambda^*$ is given in (\ref{HJBRSK}). The inequality holds also for the control policy obtained via (\ref{HJBRSK}) for every $x$.
\end{theorem}

Thus, we are able to obtain an explicit bound, for all $N$, by approximating a coupled problem with one that is not coupled.

\section{Conclusion}
   In this paper, we showed that a decentralized or mean-field control problem with interacting agents can be approached by considering a risk-sensitive version of non-interacting agents/particles or particles with structural coupling (such as symmetric), leading to a robust formulation. 
   
   In our first approach, we perturbed the noise on the infinite product space directly. This has the benefit of allowing for general cost criteria at the drawback of the finite energy condition defined in equation \eqref{eq:Novikov} that in particular does not allow for symmetric interaction.

   In the second approach, we perturbed the law of the agents on the finite product space. This eliminates the need for the finite energy condition defined in equation \eqref{eq:Novikov} and allows for symmetric interaction.  

%   The third approach is perturbing the law of the noise on the finite product space. Similarly to the second approach, we don't need the finite energy condition defined in equation \eqref{eq:Novikov}. This approach has an unnatural perturbation of $\sigma(X)v_t$, however when $\sigma$ is small the robustness bound is much tighter. 
   
   We obtained bounds on the value of the cost function in terms of the risk-sensitive cost function for the noninteracting case plus a term involving the strength of the interaction characterized by relative entropy. We emphasize the mathematical flexibility afforded when the interaction is symmetric, where in the absence of symmetry a bounded energy condition is needed whereas in the symmetric setup this is not.

\section*{Acknowledgments} We would like to gratefully acknowledge Alex Dunlap for a discussion on optimal risk sensitivity which led to Section \ref{sec:optimal-risk-sensitivity}.
\bibliographystyle{plain}
\bibliography{bibliography,SerdarBibliography}

\end{document}